\documentclass{article}

\usepackage{authblk}

\usepackage{amsmath}
\usepackage{amsthm}
\usepackage{amsfonts}
\usepackage{amssymb}
\usepackage[all]{xy}
\usepackage{url}

\DeclareMathSymbol{\rightrightarrows}  {\mathrel}{AMSa}{"13}

\def\Br{\operatorname{Br}}
\def\L{\operatorname{L}}
\def\max{\operatorname{max}}

\catcode`\@=11
\def\varholim@#1#2{\mathop{\vtop{\ialign{##\crcr
        \hfil$#1\m@th\operator@font holim$\hfil\crcr
 \noalign{\nointerlineskip\kern\ex@}#2#1\crcr
 \noalign{\nointerlineskip\kern-\ex@}\crcr}}}}
\def\hocolim{\mathpalette\varholim@\rightarrowfill@} 
\def\hoinvlim{\mathpalette\varholim@\leftarrowfill@}
\catcode`\@=\active

\newtheorem{theorem}{Theorem}
\newtheorem{lemma}[theorem]{Lemma}
\newtheorem{corollary}[theorem]{Corollary}

\theoremstyle{definition}

\newtheorem{remark}[theorem]{Remark}

\begin{document}

\title{\bf Layers and stability}
\author{J.F. Jardine\thanks{Supported by NSERC.}}

\affil{\small Department of Mathematics\\University of Western Ontario\\
  London, Ontario, Canada
}
\affil{jardine@uwo.ca}


\maketitle

\begin{abstract}
  The hierarchy associated to clusters in the HDBSCAN algorithm has layers, which are defined by cardinality. The layers define a layer subposet of the HDBSCAN hierarchy, which is a strong deformation retract and admits a stability analysis. That stability analysis is introduced here. Cardinality arguments lead to sharper results for layers than one sees for stability statements for branch points.

  This paper appeared on the author's website in January, 2021.
\end{abstract}

\section*{Introduction}

Every finite metric space $X=(X,d)$ has an associated system of partially ordered sets $P_{s}(X)$, where $s$ is a non-negative real number. This system is filtered by the systems $P_{s,k}(X)$ where $k$ is a positive integer.

The poset $P_{s}(X)$ consists of those subsets $\sigma$ of $X$ such that $d(x,y) \leq s$ for all $x,y \in  \sigma$.

The poset $P_{s,k}(X)$ consists of those subsets $\tau$ of $X$ such that each $x \in \tau$ has at least $k$ distinct neighbours $y \in X$ such that $d(x,y) \leq s$. We also require that $d(x,x') \leq s$ for any two members $x,x'$ of $\tau$.

The poset $P_{s}(X)$ is the poset of simplices for the Vietoris-Rips complex $V_{s}(X)$, and the poset $P_{s,k}(X)$ is the poset of simplices of the degree Rips comples $L_{s,k}(X)$.

Observe that $P_{s,0}(X) = P_{s}(X)$, so that the complex $L_{s,0}(X)$ is the Vietoris-Rips complex $V_{s}(X)$.
\medskip

For a fixed density parameter $k$, the path component functor $\pi_{0}$ defines an assignment $s \mapsto \pi_{0}L_{s,k}(X)$, giving a functor defined on the poset $[0,\infty]$ that takes values in sets.  The sets of path components $\pi_{0}L_{s,k}(X)$ are commonly called clusters.

This functor defines a graph $\Gamma_{k}(X)$, with vertices consisting of pairs $(s,[x])$ with $[x] \in \pi_{0}L_{s,k}(X)$. There is an edge $(s,[x]) \to (t,[y])$ if $s \leq t$ and $[x]=[y]$ in $\pi_{0}L_{t,k}(X)$. This graph is a hierarchy, or tree, which is commonly called the HDBSCAN hierarchy. It is also a poset because the edges can be composed, and this poset has a terminal object. I write $(s,[x]) \leq (t,[y])$ for edges (or morphisms) of $\Gamma_{k}(X)$ to reflect the poset structure.

A vertex $(s,[x])$ of $\Gamma_{k}(X)$ is a {\it branch point} if either $(s,[x])$ has no antecedents $(t,[y]) \leq (s,[x])$, or if $(s,[x])$ has distinct antecedents $(t,[y_{1}])$ and $(t,[y_{2}])$ for sufficiently close $t < s$.

A vertex $(t,[y])$ is a {\it layer point} if it has no antecedents, or if for all antecedents $(s,[z]) \leq (t,[y])$ with $s < t$, the set $[z]$ is strictly smaller than $[y]$ as a subset of $X$.

Every branch point is a layer point, but the converse assertion does not hold in general.
Layer points and branch points do coincide for the Vietoris-Rips system $V_{s}(X) = L_{s,0}(X)$, since the underlying system of vertex sets is constant.
\medskip

The branch points and layer points, respectively, define subposets $\Br_{k}(X)$ and $\L_{k}(X)$ of the tree $\Gamma_{k}(X)$, and there are poset inclusions
\begin{equation*}
  \Br_{k}(X) \subseteq \L_{k}(X) \subset \Gamma_{k}(X).
\end{equation*}
These subposets are themselves hierarchies.

The purpose of this note is to describe stability properties of the layer poset $\L_{k}(X)$. There is a similar investigation of stability properties of branch points in \cite{br-pts} --- the posets $\Br_{k}(X)$ of branch points and $\L_{k}(X)$ of layer points have similar properties, but the cardinality counts associated with layer points are sharper tools. 
\medskip

The first section of this paper establishes the formal properties of the poset $L_{k}(X)$ of layer points. The most important feature of $L_{k}(X)$ is that it has a calculus of least upper bounds (Lemma \ref{lem 4}), which mirrors the theory of least upper bounds for the branch point poset $\Br_{k}(X)$ that appears in \cite{br-pts}. The inclusion $\Br_{k}(X) \subset L_{k}(X)$ preserves least upper bounds. The poset morphisms
\begin{equation*}
  \Br_{k}(X) \subseteq L_{k}(X) \subset \Gamma_{k}(X)
\end{equation*}
define both $\Br_{k}(X)$ and $\L_{k}(X)$ as strong deformation retracts of the poset $\Gamma_{k}(X)$, in a way that is consistent with the inclusion $\Br_{k}(X) \subset L_{k}(X)$ --- see Lemma \ref{lem 5} and Lemma \ref{lem 6}. The retraction map $\max: \Gamma_{k}(X) \to L_{k}(X)$ is defined by setting $\max(t,[x])$ to be the maximal layer point below $(t,[x])$. The layer point $\max(t,[x])$ can also be defined to be the minimal point $(s,[z]) \leq (t,[x])$ such that $[z]=[x]$ as subsets of the set $X$.
\medskip

If $i: X \subset Y$ is an inclusion of finite metric spaces, then there is an induced poset map $i_{\ast}: L_{k}(X) \to L_{k}(Y)$, where $i_{\ast}(t,[x])$ is defined to be the maximal layer point below $(t,[i(x)])$ in $\Gamma_{k}(Y)$. 

The degree Rips stability theorem (Theorem 4 of \cite{persist-htpy}) says that there are homotopy commutative diagrams
\begin{equation}\label{eq 1}
  \xymatrix{
    L_{s,k}(X) \ar[r]^{\sigma} \ar[d]_{i} & L_{s+2r,k}(X) \ar[d]^{i} \\
    L_{s,k}(Y) \ar[r]_{\sigma} \ar[ur]^{\theta} & L_{s+2r,k}(Y)
  }
  \end{equation}
in the presence of a condition $d_{H}(X^{k}_{dis},Y^{k}_{dis}) < r$ on Hausdorff distance between spaces of $k+1$ distinct points in $X$ and $Y$.
Theorem \ref{th 8} of Section 2 says that the diagram (\ref{eq 1}) induces a homotopy commutative diagram
\begin{equation}\label{eq 2}
  \xymatrix{
    L_{k}(X) \ar[r]^{\sigma_{\ast}} \ar[d]_{i_{\ast}} & L_{k}(X) \ar[d]^{i_{\ast}} \\
    L_{k}(Y) \ar[r]_{\sigma_{\ast}} \ar[ur]^{\theta_{\ast}} & L_{k}(Y)
  }
\end{equation}
The map $i_{\ast}$ in (\ref{eq 2}) has already been defined, and all other maps in (\ref{eq 2}) are defined analogously. For example, the shift homomorphism $\sigma_{\ast}$ is defined, for a layer point $(t,[x])$, by taking $\sigma_{\ast}(t,[x])$ to be the maximal layer point below $(t+2r,[x])$. The homotopy commutativity of (\ref{eq 2}) amounts to the existence of natural relations
\begin{equation*}
  \theta_{\ast} \cdot i_{\ast} \leq \sigma_{\ast}\ \text{and}\ i_{\ast} \cdot \theta_{\ast} \leq \sigma_{\ast}.
  \end{equation*}

These relations have to be interpreted a bit carefully. If, for example, $(s.[x])$ is a layer point of $\Gamma_{k}(X)$, then $(t,[x])$ is a layer point below $(t+2r,[x])$ so that $(t,[x]) \leq \sigma_{\ast}(t,[x])$. This means that $\sigma_{\ast}(t,[x])$ is a common upper bound for $(t,[x])$ and $\theta_{\ast}i_{\ast}(t,[x])$, while $\sigma_{\ast}(t,[x])$ has the form $(u,[x])$ for some parameter value $t \leq u \leq t+2r$.

I have not yet found a good way to estimate the corresponding parameter value of $\theta_{\ast}i_{\ast}(t,[x])$ without some extra assumptions. At this level of generality, we have the same issues with locating parameter values for the points $i_{\ast}(t,[x])$ and $\theta_{\ast}(s,[y])$, relative to $t$ and $s$, respectively.
\medskip

We can sharpen these relations if the layer points are sufficiently sparse. The {\it layer parameters} are the parameters $t$ associated to the layer points $(t,[x])$ of $\Gamma_{k}(X)$. A layer parameter $t$ can have a successor $t_{+}$ and a predecessor $t_{-}$. Lemma \ref{lem 14} of this paper says that, if $r < t < t_{+}-2r$, then $i_{\ast}(t,[x]) = (s,[y])$, where $t-2r \leq s \leq t$.
Under the same assumptions, Corollary \ref{cor 15} further says that $\theta_{\ast}i_{\ast}(t,[x]) = (t,[x])$.
\medskip

Lemma \ref{lem 14} and Corollary \ref{cor 15} deal with layer points $(t,[x])$ of $\Gamma_{k}(X)$ which have enough room ``above'' them.
If $r$ is sufficiently small such that $r < t < t_{+}-2r$ for all layer parameters $t$ of $X$, then $\theta_{\ast}i_{\ast}(t,[x]) = (t,[x])$ for all layer points $(t,[x])$ of $X$, so that $\Gamma_{k}(X)$ is a retract of $\Gamma_{k}(Y)$.

This can be achieved, for example, if $X \subset Z$ is an inclusion of metric spaces, where $X$ is interpreted as a set of marked points, $r$ is chosen sufficently small that $r < t < t_{+}-2r$ for all layer parameters $t$ of $X$, and the points of $Y \subset Z$ are chosen such that $d_{H}(X^{k}_{dis},Y^{k}_{dis}) < r$ in $Z^{k}_{dis}$.
\medskip

The analysis simplifies for Vietoris-Rips complexes. In that case, $X$ and $Y$ are the vertex sets of $V_{s}(X)$ and $V_{s}(Y)$, respectively, for all $s$. Then Lemma \ref{lem 17} says that if $(s,[y])$ is a layer point of $\Gamma_{0}(Y)$ and $(t,[x])$ is a maximal layer point below $(s+2r,[\theta(y)])$, then $s \leq t \leq s+2r$. This means, for example, that every layer parameter $s$ of $\Gamma_{0}(Y)$ satisfies $t-2r \leq s \leq t$ for some layer parameter $t$ of $\Gamma_{0}(X)$. Lemma \ref{lem 14} and Lemma \ref{lem 17} together say that the layer parameters of $\Gamma_{0}(X)$ and $\Gamma_{0}(Y)$ for the respective Vietoris-Rips systems are very tightly bound, in a predictable way.


\section{Layer points}

Suppose that $X$ is a finite metric space and that $k$ is a positive integer. The functor $s \mapsto \pi_{0}L_{s,k}(X)$  has a homotopy colimit $\Gamma_{k}(X)$ having objects $(s,[x])$ with $[x] \in \pi_{0}L_{s,k}(X)$ and morphisms $(s,[x]) \to (t,[x])$ with $s \leq t$. Here, the distance parameters $s$ are positive real numbers, and hence members of the interval $[0,\infty]$.

This category $\Gamma_{k}(X)$ is a partially ordered set, and has the structure of a tree, and one writes $(s,[x]) \leq (t,[y])$ for its morphisms.
The spaces $L_{s,k}(X)$ are connected for $s$ sufficiently large, say $s \geq R$, since $X$ is a finite set.

I often write $[x]_{s}$ for $[x] \in \pi_{0}L_{s,k}$. The path component $[x]_{s}$ is a subset of the vertices of $L_{s,k}(X)$. There is a relation $(s,[x]) \leq (t,[y])$ if and only if $s \leq t$ and $[x]_{s} \subset [y]_{t}$ as subsets of $X$.

        A {\it branch point} in the tree $\Gamma_{k}(X)$ is a vertex $(t,[x])$ such that either of following two conditions hold:
    \begin{itemize}
    \item[1)] there is an $s_{0} < t$ such that for all $s_{0} \leq s < t$ there are distinct vertices $(s,[x_{0}])$ and $(s,[x_{1}])$ with $(s,[x_{0}]) \leq (t,[x])$ and $(s,[x_{1}]) \leq (t,[x])$, or
    \item[2)] there is no relation $(s,[y]) \leq (t,[x])$ with $s < t$.
\end{itemize}
The second condition means that the path component $[x]$ does not have a representative in  $L_{s,k}(X)$ for $s < t$.    
Write $\Br_{k}(X)$ for the subposet of $\Gamma_{k}(X)$, which is defined by the branch points.
\medskip

A {\it layer point} of $\Gamma_{k}(X)$ is a vertex $(t,[x])$ such that one of the following two conditions hold:
\begin{itemize}
\item[1)] if there is a relation $(s,[y]) \leq (t,[x])$ with $s < t$, then $[y]_{s}$ is a proper subset of $[x]_{t}$, equivalently there is a proper inequality $\vert [y]_{s} \vert < \vert [x]_{t} \vert$ in cardinality, or
\item[2)] there is no relation  $(s,[y]) \leq (t,[x])$ with $s < t$.
  \end{itemize}
The layer points form a subposet $\L_{k}(X)$ of $\Gamma_{k}(X)$.

\begin{remark}
  There is a maximal finite subsequence
  \begin{equation*}
    0 \ne t_{1} < \dots < t_{p}
  \end{equation*}
  of positive real numbers $t_{j}$, which are the distances between vertices of \begin{equation*}
    L_{k,t_{p}}(X) = L_{k,\infty}(X).
  \end{equation*}
  The numbers $t_{i}$ are the {\it phase change} numbers for the system $L_{\ast,k}(X)$. Observe that the vertices of $L_{k,t_{i}}(X)$ and $L_{k,t_{i+1}}(X)$ could coincide.

  We can find the layer points for $\Gamma_{k}(X)$ by induction on $i$, starting with the observation that all points $(t_{1},[z])$ are layer points. If $[x] \in \pi_{0}L_{t_{i},k}(X)$, then $[x] \cap L_{t_{i-1},k}(X)_{0}$ is a disjoint union of path components $[y]$. This intersection could be empty, in which case $(t_{i},[x])$ is a layer point. Otherwise, $(t_{i},[x])$ is a layer point if all $[y] \subset [x] \cap L_{t_{i-1},k}(X)_{0}$ satisfy $\vert [y] \vert < \vert [x] \vert$.
  \end{remark}

\begin{lemma}\label{lem 2}
  All branch points are layer points, and so there are poset inclusions
\begin{equation*}
  \Br_{k}(X) \subseteq \L_{k}(X) \subset \Gamma_{k}(X).
  \end{equation*}
\end{lemma}

\begin{proof}
  Suppose that condition 1) holds for the branch point $(t,[x])$: there is an $s_{0} < t$ that for all $s_{0} \leq s <t$ there are distinct points $(s,[x_{0}])$ and $(s,[x_{1}])$ such that $(s,[x_{i}]) \leq (t,[x])$.

  If $(s,[z]) \leq (t,[x])$ then $[z]$ is one of multiple path components $[v]_{s}$ of $L_{s,k}(X)$ that map to $[x]_{t}$ in $L_{t,k}(X)$. All such components are proper subsets of $[x]_{t}$. 
\end{proof}

Recall that $L_{0,s}(X)$ is the Vietoris-Rips complex $V_{s}(X)$, and that the elements of $X$ are the vertices of the Vietoris-Rips complex $V_{s}(X)$. All complexes $V_{s}(X)$ have the same vertices, namely the set  $X$.

\begin{lemma}\label{lem 3}
Every layer point of $\Gamma_{0}(X)$ is a branch point, so that $\Br_{0}(X) = L_{0}(X)$.
\end{lemma}

\begin{proof}
The underlying sets of vertices for $V_{s}(X)$ and $V_{t}(X)$ coincide. Thus, if $(t,[x])$ is a layer point of $\Gamma_{0}(X)$ and $s < t$, then the collection $[y]$ of components of $V_{s}(X)$ that map to $[x]$ in $V_{t}(X)$ is non-empty and satisfies $\sqcup\ [y]_{s} = [x]_{t}$. There are multiple such summands $[y]_{s}$, since $(t,[x])$ is a layer point, so that all inclusions $[y]_{s} \subset [x]_{t}$ are proper. In particular, there are distinct elements $(s,[y])$ and $(s,[y'])$ below $(t,[x])$.
\end{proof}

Suppose that $(s,[x])$ and $(t,[y])$ are vertices of the graph $\Gamma_{k}(X)$. There is a unique smallest vertex $(u,[z])$ which is an upper bound for both $(s,[x])$ and $(t,[y])$ in $\Gamma_{k}(X)$. The number $u$ is the smallest parameter (necessarily a phase change number) such that $[x]_{u}=[y]_{u}$ in $\pi_{0}L_{u,k}(X)$, and so $[z]_{u}=[x]_{u}=[y]_{u}$.
In this case, 
one writes
\begin{equation*}
  (s,[x]) \cup (t,[y]) = (u,[z]).
  \end{equation*}
The vertex $(u,[z])$ is the {\it least upper bound} (or join) of $(s,[x])$ and $(t,[y])$.

Every finite collection of points $(s_{1},[x_{1}]), \dots ,(s_{p},[x_{p}])$ has a least upper bound
\begin{equation*}
  (s_{1},[x_{1}]) \cup \dots \cup (s_{p},[x_{p}])
\end{equation*}
in the tree $\Gamma_{k}(X)$. 

We know from \cite{br-pts} that the least upper bound of two branch points is a branch point, and we have an analogous result for layer points:

\begin{lemma}\label{lem 4}
  The least upper bound $(u,[z])$ of layer points $(s,[x])$ and $(t,[y])$ is a layer  point.
\end{lemma}

\begin{proof}
If there is a number $v$ such that $s,t<v<u$, then $(v,[x])$ and $(v,[y])$ are distinct because $(u,[z])$ is a least upper bound. This implies that $L_{v,k}(X)$ has distinct path components $[w]$ which map to $[z]$ in $\pi_{0}L_{u,k}(X)$. It follows that $(u,[z])$ is a branch point, and is therefore a layer point by L:emma \ref{lem 2}.

  Otherwise, $s=u$ or $t=u$, in which case $(u,[z]) = (s,[x])$ or $(u,[z])=(t,[y])$. In either case, $(u,[z])$ is a layer point.
\end{proof}

Lemma \ref{lem 4} implies that every collection of layer points $(s_{1},[x_{1}]), \dots , (s_{p},[x_{p}])$ has a least upper bound
\begin{equation*}
  (s_{1},[x_{1}]) \cup \dots \cup (s_{p},[x_{p}])
\end{equation*}
in $\L_{k}(X)$. The maximal (or terminal) element of $L_{k}(X)$ is the least upper bound of all members of $L_{k}(X)$.

It follows from Lemma \ref{lem 4} and the corresponding result for branch points of \cite{br-pts} that the poset inclusions
\begin{equation*}
  \Br_{k}(X) \subseteq \L_{k}(X) \subset \Gamma_{k}(X)
\end{equation*}
preserve least upper bounds.

\begin{lemma}\label{lem 5}
Every vertex $(s,[x])$ of $\Gamma_{k}(X)$ has a unique largest layer point $(t,[y])$ such that $(t,[y]) \leq (s,[x])$. In this case, $[y]_{t} = [x]_{s}$.
\end{lemma}

\begin{proof}
  There is a smallest phase change number $t$ such that there is a relation $(t,[y]) \leq (s,[x])$ with $[y]_{t} = [x]_{s}$. The corresponding point $(t,[y])$ is a layer point, by the minimality of the phase change number $t$.

  The point $(t,[y])$ is also an upper bound on the layer points below $(s,[x])$, since $[y]_{t}=[x]_{s}$: if $(u,[z])$ is a layer point below $(s,[x])$, then $z \in [y]_{t}$ and $u \leq t$ since otherwise $(u,[z])$ is not a layer point.
  \end{proof}

The first statement of Lemma \ref{lem 5} is also a corollary of Lemma \ref{lem 4}: take the least upper bound of all layer points below $(s,[x])$.

\begin{lemma}\label{lem 6}
The poset inclusion $\L_{k}(X) \subset \Gamma_{k}(X)$ has an inverse
\begin{equation*}
  \max: \Gamma_{k}(X) \to \L_{k}(X),
\end{equation*}
up to homotopy, and $\L _{k}(X)$ is a strong deformation retract of $\Gamma_{k}(X)$.
\end{lemma}

\begin{proof}
Every vertex $(s,[x])$ of $\Gamma_{k}(X)$ has a unique maximal layer point $(s_{0},[x_{0}])$ such that $(s_{0},[x_{0}]) \leq (s,[x])$, by Lemma \ref{lem 5}. Set
\begin{equation*}
  \max(s,[x]) = (s_{0},[x_{0}]).
\end{equation*}
The maximality condition implies that the function $\max$ preserves the ordering. The composite
$\max \cdot \alpha$ is the identity on $\L_{k}(X)$, and the relations $(s_{0},[x_{0}]) \leq (s,x)$ define a homotopy $\alpha \cdot \max \leq 1$ that restricts to the identity on $\L_{k}(X)$.
\end{proof}

\begin{remark}
  Lemma 5 of \cite{br-pts} says that every $(s,[x])$ has a unique maximal branch point $(s_{1},[x_{1}])$ such that $(s_{1},[x]_{1}) \leq (s,[x])$. The branch point $(s_{1},[x_{1}])$ is a layer point by Lemma \ref{lem 2}, so that there are relations.
  \begin{equation*}
    (s_{1},[x_{1}]) \leq (s_{0},[x_{0}]) \leq (s,[x]),
  \end{equation*}
  which are natural in points $(s,[x])$ of $\Gamma_{k}(X)$.

  It follows that the poset inclusions
  \begin{equation*}
    \Br_{k}(X) \subseteq \L_{k}(X) \subset \Gamma_{k}(X)
  \end{equation*}
  define strong deformation retractions, and that the respective contracting homotopies are compatible.

  Recall from Lemma \ref{lem 3} that $\Br_{0}(X) = \L_{0}(X)$, so that the discussion simplifies for Vietoris-Rips complexes.
  \end{remark}

\section{Stability}

The general setup for stability of degree Rips complexes is the following: we suppose given finite metric spaces $X \subset Y$ such that the Hausdorff distance between the corresponding spaces $X_{dis}^{k}$ and $Y_{dis}^{k}$ of sets of $k+1$ distinct elements in $X$ and $Y$ respectively satisfies $d_{H}(X^{k}_{dis},Y^{k}_{dis}) < r$, where $r$ is a fixed non-zero positive real number.

Under these assumptions, the degree Rips stability theorem (Theorem 4 of \cite{persist-htpy}) says that there are homotopy commutative diagrams
\begin{equation}\label{eq 3}
  \xymatrix{
    L_{s,k}(X) \ar[r]^{\sigma} \ar[d]_{i} & L_{s+2r,k}(X) \ar[d]^{i} \\
    L_{s,k}(Y) \ar[r]_{\sigma} \ar[ur]^{\theta} & L_{s+2r,k}(Y)
  }
  \end{equation}
Applying the path component functor $\pi_{0}$ gives commutative diagrams
\begin{equation}\label{eq 4}
  \xymatrix{
    \pi_{0}L_{s,k}(X) \ar[r]^{\sigma} \ar[d]_{i} & \pi_{0}L_{s+2r,k}(X) \ar[d]^{i} \\
    \pi_{0}L_{s,k}(Y) \ar[r]_{\sigma} \ar[ur]^{\theta} & \pi_{0}L_{s+2r,k}(Y)
  }
  \end{equation}
and there is an induced commutative diagram of hierarchies
\begin{equation}\label{eq 5}
  \xymatrix{
    \Gamma_{k}(X) \ar[r]^{\sigma} \ar[d]_{i} & \Gamma_{k}(X) \ar[d]^{i} \\
    \Gamma_{k}(Y) \ar[r]_{\sigma} \ar[ur]^{\theta} & \Gamma_{k}(Y)
  }
\end{equation}
Here,
\begin{equation*}
  \begin{aligned}
    &i((s,[x])) = (s,[i(x)]),\\
    &\sigma((s,[x])) = (s+2r,[\sigma(x)]),\ \text{and}\\
    &\theta((s,[y]) = (s+2r,[\theta(y)]).
  \end{aligned}
\end{equation*}

Write $i_{\ast}: \L_{k}(X) \to \L_{k}(Y)$ for the composite poset morphism
\begin{equation*}
  \L_{k}(X) \subset \Gamma_{k}(X) \xrightarrow{i_{\ast}} \Gamma_{k}(Y) \xrightarrow{\max} \L_{k}(Y)
  \end{equation*}
This map takes a layer point $(s,[x])$ to the maximal layer point below $(s,[i(x)])$. 

Poset morphisms $\theta_{\ast}: \L_{k}(Y) \to \L_{k}(X)$ and $\sigma_{\ast}: \L_{k}(X) \to \L_{k}(X)$ are similarly defined, respectively, by the poset morphisms $\theta: \Gamma_{k}(Y) \to \Gamma_{k}(X)$ and the shift morphism $\sigma: \Gamma_{k}(X) \to \Gamma_{k}(X)$.
\medskip

\noindent
1)\ Consider the poset maps
\begin{equation*}
  \L_{k}(X) \xrightarrow{i_{\ast}} \L_{k}(Y) \xrightarrow{\theta_{\ast}} \L_{k}(X).
\end{equation*}

If $(s,[x])$ is a layer point for $X$, choose maximal layer points $(s_{0},[x_{0}]) \leq (s,[i(x)]$, $(s_{1},[x_{1}]) \leq (s_{0}+2r,[\theta(x_{0})])$ and $(v,[y]) \leq (s+2r,[x])$ below the respective objects.

Then $\theta_{\ast}i_{\ast}(s,[x]) = (s_{1},[x_{1}])$, and there is a natural relation
\begin{equation*}
  \theta_{\ast}i_{\ast}(s,[x]) = (s_{1},[x_{1}]) \leq (v,[y]) = \sigma_{\ast}(s,[x]) 
\end{equation*}
by a maximality argument.
We therefore have a homotopy of poset maps
\begin{equation}\label{eq 6}
  \theta_{\ast}i_{\ast} \leq \sigma_{\ast}: \L_{k}(X) \to \L_{k}(X).
\end{equation}

\noindent
2)\ Similarly, if $(t,[y])$ is a layer point of $Y$, then
\begin{equation*}
  i_{\ast}\theta_{\ast}(t,[y]) \leq \sigma_{\ast}(t,[y]),
\end{equation*}
giving a homotopy
\begin{equation}\label{eq 7}
  i_{\ast}\theta_{\ast} \leq \sigma_{\ast}: \L_{k}(Y) \to \L_{k}(Y).
  \end{equation}

There are relations
\begin{equation}
  (s,[x]) \leq \sigma_{\ast}(s,[x]) \leq (s+2r,[x])
\end{equation}
for branch points $(s,[x])$.
It follows that the poset map $\sigma_{\ast}: \L_{k}(X) \to \L_{k}(X)$ is homotopic to the identity on $\L_{k}(X)$.
\medskip

The construction of the poset maps $i_{\ast}$, $\theta_{\ast}$ and $\sigma_{\ast}$, together with the relations (\ref{eq 6}) and (\ref{eq 7}), complete the construction/proof of the following result:

\begin{theorem}\label{th 8}
Suppose that $X \subset Y$ is an inclusion of finite metric spaces, and that $d_{H}(X^{k}_{dis},Y^{k}_{dis}) < r$. Then there is a homotopy commutative diagram
\begin{equation}\label{eq 9} 
  \xymatrix{
  L_{k}(X) \ar[r]^{\sigma_{\ast}} \ar[d]_{i_{\ast}} & L_{k}(X) \ar[d]^{i_{\ast}} \\
  L_{k}(Y) \ar[ur]^{\theta_{\ast}} \ar[r]_{\sigma_{\ast}} & L_{k}(Y)
    }
\end{equation}
that relates the layer posets $\L_{k}(X)$ and $\L_{k}(Y)$ of the spaces $X$ and $Y$, respectively.
\end{theorem}

\begin{remark}
  The element $\sigma_{\ast}(s,[x]) = (t,[x])$ is close to $(s,[x])$ in the sense that there are relations
  \begin{equation*}
    (s,[x]) \leq (t,[x]) \leq (s+2r,[x])
    \end{equation*}
so that
  $0 \leq t-s \leq 2r$. Thus, the layer points $(s,[x])$ and $\theta_{\ast}i_{\ast}(s,[x])$ have a common upper bound, namely $\sigma_{\ast}(s,[x])$, which is close to $(s,[x])$.

If $(t,[y])$ is a layer point of $\Gamma_{k}(Y)$, the layer point $\sigma_{\ast}(t,[y]) \leq (t+2r,[y])$ is similarly an upper bound for $(t,[y])$ and $i_{\ast}\theta_{\ast}(t,[y])$, and is close to $(t,[y])$.

The subobject of $\L_{k}(X)$ consisting of all layer points of the form $(s,[x])$ as $s$ varies has an obvious notion of distance: the distance between points $(s,[x])$ and $(t,[x])$ is $\vert t-s \vert$.  
\end{remark}

Suppose that
\begin{equation*}
  0 < t_{1} < \dots < t_{k}
  \end{equation*}
are the phase change numbers for the system $L_{s,k}(X)$. 

The assumption that $d_{H}(X^{k}_{dis},Y^{k}_{dis}) < r$ forces the function
\begin{equation*}
  \pi_{0}L_{s,k}(X) \to \pi_{0}L_{s,k}(Y)
\end{equation*}
to be surjective if $s \geq r$.

       \begin{lemma}\label{lem 10}
         Suppose, that $y_{1},y_{2} \in Y$ have elements $\theta(y_{1}), \theta(y_{2}) \in X$ such that $d(y_{i},\theta(y_{i})) < r$. Then $d(y_{1},y_{2})$ is in the interval $(t-2r,t+2r)$, where $t=d(\theta(y_{1}),\theta(y_{2}))$.
       \end{lemma}

       \begin{proof}
         We shall assume that $t-2r > 0$.

         Consider the picture
         
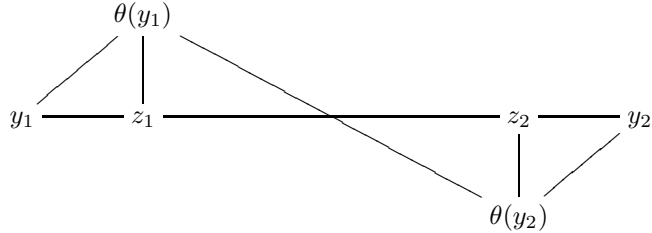
\begin{equation*}
         \xymatrix{
           & \theta(y_{1}) \ar@{-}[ddrrrr] \ar@{-}[d] \\
           y_{1} \ar@{-}[ur] \ar@{-}[r] & z_{1} \ar@{-}[rrrr] &&&& z_{2} \ar@{-}[r] & y_{2} \\
           &&&&& \theta(y_{2}) \ar@{-}[u] \ar@{-}[ur] \\
          }
         \end{equation*}
         Suppose that $v$ is the point of intersection of the lines $(z_{1},z_{2})$ and $(\theta(y_{1}),\theta(y_{2}))$. Then
         \begin{equation*}
           d(\theta(y_{1}),\theta(y_{2})) \geq d(z_{1},z_{2}) = d(z_{1},v) + d(v,z_{2}) \geq d(y_{1},y_{2}) - 2r.
           \end{equation*}
The assertion that $d(\theta(y_{1}),\theta(y_{2})) < d(y_{1},y_{2}) +2r$ is a simple application of the triangle inequality.
       \end{proof}

       \begin{corollary}\label{cor 11}
All phase change numbers $s$ for $Y$ lie in intervals $(t-2r,t+2r)$ around phase change numbers $t$ of $X$.
         \end{corollary}

       There is a finite collection of numbers $t$ such that $(t,[x])$ is a layer point for $\Gamma_{k}(X)$. Say that such numbers $t$ are the layer parameters for $X$.
Each layer parameter is a phase change number.
\medskip

       Observe that the inclusions $\sigma: L_{s,k}(X) \subseteq L_{t,k}(X)$ for $s \leq t$ induce inclusions $[x]_{s} \subset [x]_{t}$ for all vertices $x$ of $L_{s,k}(X)$.

       Recall from the proof of Lemma \ref{lem 5} that the maximal layer point below $(s,[x])$ can be constructed by finding the smallest phase change number $t$ such that there is a relations $(t,[u]) \leq (s,[x])$ such that $[u]_{t}=[x]_{s}$ as subsets of $X$.

       \begin{lemma}\label{lem 12}
         Suppose that $s < t$ and there are no layer points of the form $(u,[x])$ in $\Gamma_{k}(X)$, where $s < u \leq t$. Then the induced function
         \begin{equation*}
          \sigma_{\ast}: \pi_{0}L_{s,k}(X) \to \pi_{0}L_{t,k}(X)
         \end{equation*}
         is a bijection.
         \end{lemma}

       \begin{proof}
         We can assume that $L_{t,k}(X) \ne \emptyset$, for otherwise $L_{s,k}(X) = L_{t,k}(X) = \emptyset$.

         Suppose that $(t,[x]) \in \Gamma_{k}(X)$ and that $(u,[y])$ is a maximal layer point with $(u,[y]) \leq (t,[x])$. Then $u \leq s$ and the relations $(u,[y]) \leq (s,[y]) \leq (t,[x])$ force $[y]_{s} = [x]_{t}$. In particular, the function $\sigma_{\ast}$ is surjective.
       
         If $[y_{1}],[y_{s}] \in \pi_{0}L_{s,k}(X)$ have the same image $[x] \in \pi_{0}L_{t,k}(X)$, then $[y_{1}]_{s} = [x]_{t} = [y_{2}]_{s}$ as subsets of $X$, so that $[y_{1}] = [y_{2}]$ in $\pi_{0}L_{s,k}(X)$, and so $\sigma_{\ast}$ is injective.
         \end{proof}

Given a layer parameter $t$ for $X$, write $t_{+}$ for the smallest layer parameter of $X$ with $t < t_{+}$, and write $t_{-}$ for the largest layer parameter of $X$ with $t_{-} < t$.

\begin{lemma}\label{lem 13}
  Suppose that $t$ is a layer parameter for $X$ such that $r < t < t_{+}-2r$.
  Then the function
  $i : \pi_{0}L_{t,k}(X) \to \pi_{0}L_{t,k}(Y)$ is a bijection.
\end{lemma}

\begin{proof}
  The diagram
  \begin{equation*}
    \xymatrix{
      \pi_{0}L_{t,k}(X) \ar[r]^-{\cong} \ar[d]_{i} & \pi_{0}L_{t+2r,k}(X) \\
      \pi_{0}L_{t,k}(Y) \ar[ur]_{\theta}
    }
  \end{equation*}
  commutes, and the displayed function is a bijection by Lemma \ref{lem 12},
  so the function $i$ is injective. The surjectivity of $i$ follows from the assumption $t >r$.
\end{proof}

\vfill\eject

\begin{lemma}\label{lem 14}
  Suppose that $(t,[x])$ is a layer point of $\Gamma_{k}(X)$ with $r < t < t_{+}-2r$, and suppose that $(s,[y])$ is a maximal layer point below $(t,[i(x)])$ in $\Gamma_{k}(Y)$.
  Then $t-2r \leq s \leq t$.
  \end{lemma}

\begin{proof}
  Suppose that $s < t-2r$.

  The map $i_{\ast}: \pi_{0}L_{t,k}(X) \to \pi_{0}L_{t,k}(Y)$ is a bijection by Lemma \ref{lem 13} and $i_{\ast}([x]) = i_{\ast}([\theta(y)]) = [i(x)]$ in $\pi_{0}L_{t,k}(Y)$. It follows that there is a commutative diagram of functions
\begin{equation*}
  \xymatrix{
    & [\theta(y)]_{s+2r} \ar[r]^-{\sigma} & [x]_{t} \ar[d]^{i} \\ 
           [y]_{s} \ar[ur]^{ \theta} \ar[rr]_{\sigma}^{\cong}
           && [i(x)]_{t}
           }
\end{equation*}
in which the map $i: [x]_{t} \to [i(x)]_{t}$ is a monomorphism since it is a subobject of a monomorphism of vertices.

The functions $i$ and $\sigma \cdot \theta$ are bijections, and so $\sigma: [\theta(y)]_{s+2r} \to [x]_{t_{j}}$ is an epimorphism. This function $\sigma$ is also a monomorphism, since it is a subobject of the monomorphism of vertices $L_{s+2r,k}(X)_{0} \to L_{t,k}(X)_{0}$.

It follows that the function $\sigma: [\theta(y)]_{s+2r} \to [x]_{t}$ is a bijection, so that $(t,[x])$ is not a layer point.
\end{proof}

  \begin{corollary}\label{cor 15}
    Suppose that $(t,[x])$ is a layer point for $\Gamma_{k}(X)$ such that $r< t < t_{+}-2r$. Then we have
    \begin{equation*}
  \theta_{\ast}i_{\ast}(t,[x]) = (t,[x]).
\end{equation*}
    \end{corollary}

  \begin{proof}
  Suppose that $(s,[z])$ is a maximal layer point below $(t,[i(x)])$ in $\Gamma_{k}(Y)$. Then $t-2r \leq s \leq t$ by Lemma \ref{lem 14}, so that $t \leq s+2r \leq t+2r < t_{+}$.

  The layer point $(t,[x])$ is a maximal layer point below $(t+2r,[x])$, since $t+2r < t_{+}$, so that $[x]_{t} = [x]_{t+2r}$. The layer point $\theta_{\ast}(s,[z])$ is the maximal layer point below $(s+2r,[\theta(z)])$, and the relation
  \begin{equation*}
    (s+2r,[\theta(z)]) \leq (t+2r,[x])
  \end{equation*}
  implies that $\theta(z) \in [x]_{t+2r} = [x]_{s+2r}$, so that $x \in [\theta(z)]_{s+2r}$.
  It follows that the maximal layer point below $(s+2r,[\theta(z)])$ must also be the maximal layer point below $(t+2r,[x])$, which is $(t,[x])$. 
\end{proof}

\begin{lemma}\label{lem 16}
Suppose that $(s,[y])$ is a layer point of $\Gamma_{k}(Y)$, and that $s < s_{+}-2r$. Suppose that $(t,[z])$ is a maximal layer point below $(s+2r,[\theta(y)])$. Then $s \leq t \leq s+2r$.
  \end{lemma}

\begin{proof}
  Suppose that $t<s$.
  
  The map $\sigma: \pi_{0}L_{s,k}(Y) \to \pi_{0}L_{s+2r,k}(Y)$ is a bijection, since $\Gamma_{k}(Y)$ has no layer parameters in the interval $(s,s+2r]$, by assumption and Lemma \ref{lem 12}. It follows that the map $\theta: \pi_{0}L_{s,k}(Y) \to \pi_{0}L_{s+2r}(X)$ is a monomorphism.

Then $\theta([y])=\theta([i(z)]$ implies that $[y]_{s}=[i(z)]_{s}$, so the diagram
    \begin{equation*}
      \xymatrix{
        [i(z)]_{t} \ar[rr]^{\sigma} \ar[dr]^{\theta}
        && [y]_{s} \ar[r]^{\sigma} \ar[dr]^{\theta} & [y]_{s+2r} \\
        [z]_{t} \ar[u]^{i} \ar[r]_{\sigma} & [z]_{t+2r} \ar[rr]_{\sigma} && [\theta(y)]_{s+2r} \ar[u]_{i}
      }
    \end{equation*}
    commutes.

The commutativity of the triangle on the right implies that $\theta: [y]_{s} \to [\theta(y)]_{s+2r}$ is a monomorphism.

The function $\sigma: [z]_{t} \to [\theta(y)]_{s+2r}$ a bijection, so
$\theta: [y]_{s} \to [\theta(y)]_{s+2r}$ is a bijection.

The composite
\begin{equation*}
  [z]_{t} \xrightarrow{i} [i(z)]_{t} \xrightarrow{\sigma} [y]_{s}
\end{equation*}
  a bijection, so $\sigma: [i(z)]_{t} \to [y]_{s}$ is a bijection, and it follows that $(s,[y])$ is not a layer point.
\end{proof}

  The analysis of the morphism
\begin{equation*}
  \theta_{\ast}: V(Y) = \L_{0}(Y) \to \L_{0}(X) = V(X)
\end{equation*}
for Vietoris-Rips complexes is sharper, because all complexes $V_{s}(Y)$ share the same set of vertices, namely $Y$. In this case, we have a stronger version of Lemma \ref{lem 16}, with a very different argument.

\begin{lemma}\label{lem 17}
         Suppose that $(s,[y])$ is a layer point of $\Gamma_{0}(Y)$, and that $(t,[z])$ is a maximal layer point of $\Gamma_{0}(X)$ below $(s+2r,[\theta(y)])$. Then $s \leq t \leq s+2r$.
         \end{lemma}

\begin{proof}

  The sets $[z]_{t}$ and $[\theta(y)]_{s+2r}$ have the same cardinality, and so $\theta(y) \in [z]_{t}$.

       Consider the collection of elements $[u] \in \pi_{0}V_{t-2r}(Y)$ which map to $[z]_{t} = [\theta(y)]_{t}$ in $\pi_{0}V_{t}(X)$.
       Then $\theta^{-1}([z]_{t}) = \sqcup\ [u]$ as a subset of the vertices $Y$ of $V_{t-2r}(Y)$, and $y \in [u]$ for some $[u]$. All such components $[u]$ map to the same path component $[y]_{t}$ in $V_{t}(X)$.

       In the diagram
       \begin{equation*}
         \xymatrix{
           \theta^{-1}([z]_{t}) \ar[r] \ar[d]
           & \theta^{-1}([\theta(y)]_{s+2r}) \ar[r] \ar[d] & Y \ar[d]^{\theta} \\
                 [z]_{t} \ar[r]_{\cong} & [\theta(y)]_{s+2r} \ar[r] & X
         }
         \end{equation*}
both squares are pullbacks, so
       the function
       \begin{equation*}
         \theta^{-1}([z]_{t}) \to \theta^{-1}([\theta(y)]_{s+2r}
         \end{equation*}
       is a bijection. 
       
       Suppose that $t < s$. Then
       \begin{equation*}
         \theta^{-1}([z]_{t}) = \sqcup\ [u] \subset [y]_{t} \subset [y]_{s} \subset \theta^{-1}([\theta(y)]_{s+2r})
         \end{equation*}
while       
$\theta^{-1}([z]_{t}) = \theta^{-1}([\theta(y)]_{s+2r})$ as subsets of $Y$.

It follows that $[y]_{t} = [y]_{s}$, so that $(s,[y])$ is not a layer point.
\end{proof}

  Lemma \ref{lem 17} and Lemma \ref{lem 14} together impose rather tight constraints on the layer points of $\Gamma_{0}(Y)$, in relation to those of $\Gamma_{0}(X)$. Recall that the comparison $\Gamma_{0}(X) \to \Gamma_{0}(Y)$ arises from applying path component functors to the comparison $V_{\ast}(X) \to V_{\ast}(Y)$. In this case, $d_{H}(X,Y)=r$ is the bound on Hausdorff distance which leads to the interleaving diagrams (\ref{eq 3}), (\ref{eq 4}) and (\ref{eq 5}).

  To repeat the statement of Lemma \ref{lem 17}, suppose that $(s,[y])$ is a layer point for $\Gamma_{0}(Y)$, and suppose that $(t,[x])$ is a maximal layer point below $(s+2r,[\theta(y)])$. Then $s \leq t \leq s+2r$.

  It follows, in particular, that all layer points of $\Gamma_{0}(Y)$ are in the intervals $[t-2r,t]$ corresponding to layer points $(t,[x])$ of $\Gamma_{0}(X)$.

\nocite{clusters}

\bibliographystyle{plain}
\bibliography{spt}

\end{document}